\documentclass[11pt]{amsart}
\usepackage[margin=3cm]{geometry}

\title{Unique Ergodicity of Non-Linear Filters via Reachability and Uniform Weak Continuity} 

\author{Yunus Emre Demirci}
\address{Queen's University, Department of Mathematics and Statistics, Kingston, Ontario, Canada}
 \email{21yed@queensu.ca}
 
\author{Serdar Yüksel}
\address{Queen's University, Department of Mathematics and Statistics, Kingston, Ontario, Canada} 
 \email{yuksel@queensu.ca}  
 
\usepackage{graphicx}%
\usepackage{multirow}%
\usepackage{amsmath,amssymb,amsfonts}%
\usepackage{amsthm}%
\usepackage{mathrsfs}%
\usepackage[title]{appendix}%
\usepackage{xcolor}%
\usepackage{textcomp}%
\usepackage{manyfoot}%
\usepackage{booktabs}%
\usepackage{algorithm}%
\usepackage{algorithmicx}%
\usepackage{algpseudocode}%
\usepackage{listings}%
\usepackage{amsmath,nicefrac,upgreek,mathtools,enumerate,url}
\usepackage{amsthm,amsmath,amsfonts,amssymb}
\usepackage[numbers]{natbib}
\usepackage{enumitem}
\usepackage{verbatim}

\newcommand{\bea}{\begin{eqnarray}}
\newcommand{\ena}{\end{eqnarray}}
\newcommand{\beas}{\begin{eqnarray*}}
\newcommand{\enas}{\end{eqnarray*}}
\newcommand{\beq}{\begin{equation}}
\newcommand{\enq}{\end{equation}}

\newcommand{\BL}{\operatorname{BL}_1}

\newcommand\norm[1]{\left\lVert#1\right\rVert}

\newcommand{\T}{\mathcal{T}}
\newcommand{\Q}{\mathcal{Q}}
\newcommand{\Z}{\mathcal{Z}}
\newcommand{\X}{\mathbb{X}}

\newcommand{\B}{\mathbb{B}}
\newcommand{\PP}{\mathbb{P}}
\usepackage{bbm}
\newcommand{\1}{\mathbbm{1}}
\usepackage{color}

\newtheorem{theorem}{Theorem}
\newtheorem{remark}{Remark}%

\newtheorem{definition}{Definition}%

\newtheorem{assumption}{Assumption}%
\newtheorem{corollary}{Corollary}%
\newtheorem{lemma}{Lemma}%

\usepackage{amsmath,nicefrac,upgreek,mathtools,enumerate,url}
\usepackage{amsthm,amsmath,amsfonts,amssymb}
\usepackage[numbers]{natbib}
\usepackage{enumitem}

\usepackage{color}

\newcommand{\yed}[1]{{\color{black} #1}}


\begin{document}

\begin{abstract} 
We present a reachability based approach to establish unique ergodicity of non-linear filter processes where state space of a
hidden Markov model is a compact Polish metric space and the observation space is
a Polish metric space. We also establish a weak convergence result on occupation
measures under such a reachability condition. Our conditions, which are explicit, are complementary to those based on filter stability as demonstrated in examples.
\end{abstract}

\maketitle
\section{Introduction}
Nonlinear filtering plays a central role in partially observable stochastic systems. In this paper, we focus on filter processes with compact state spaces and show that under mild and verifiable assumptions, they admit a unique invariant measure and exhibit weak convergence of occupation measures.
\section{Notation and preliminaries}
Let $(\mathbb{X}, d)$ denote a compact metric space which is the state space of a partially observable Markov process. 
Let $\mathbb{B}(\mathbb{X})$ be its Borel $\sigma$-field, and denote by $\mathcal{Z}:=\mathbb{P}(\mathbb{X})$ the set of probability measures on $(\mathbb{X}, \mathbb{B}(\mathbb{X}))$ under the weak topology. Let $\mathbb{P}(\mathcal{Z})$ denote the set of probability measures on $\mathcal{Z}$, equipped with the weak topology. Let $\mathbb{C}(\mathbb{X})$ be the set of all continuous, bounded functions on $\mathbb{X}$. We let $\mathcal{T}$ be the transition kernel of the model which is a stochastic kernel from $\PP(\mathbb{X})$ to $\PP(\mathbb{X})$. Here and throughout the paper $\mathbb{N}$ denotes the set of positive integers.

Let $\mathbb{Y}$ be a Polish space denoting the observation space of the model, and let the state be observed through an observation channel $\Q$. $\Q(\cdot \mid x)$ is a probability measure on $\mathbb{Y}$ for every $x \in \mathbb{X}$, and $\Q(A \mid \cdot): \mathbb{X} \rightarrow[0,1]$ is a Borel measurable function for every $A \in \B(\mathbb{Y})$. 
The update rules of the system are determined by $\mathcal{T}$ and $\Q$:
$$
\operatorname{Pr}\left(\left(X_0, Y_0\right) \in B\right)=\int_B \mu\left(d x_0\right) \Q\left(d y_0 \mid x_0\right), \quad B \in \B(\mathbb{X} \times \mathbb{Y})
$$
where $\mu$ is the (prior) distribution of the initial state $X_0$, and
$$
\operatorname{Pr}\left(\left(X_n, Y_n\right) \in B \mid(X, Y)_{[0, n-1]}=(x, y)_{[0, n-1]}\right)=\int_B \mathcal{T}\left(d x_n \mid x_{n-1}\right) \Q\left(d y_n \mid x_n\right),
$$
$B \in \B(\mathbb{X} \times \mathbb{Y})$, $n \in \mathbb{N}$.
By standard realization theorems (e.g., \cite[Lemma F]{aumann1961mixed}), this process admits a functional representation:
\begin{eqnarray}\label{updateEq}
 X_{k+1} = H(X_k, W_k) \label{updateEq1}, \qquad
Y_{k} = G(X_k, V_k) \label{updateEq2} 
\end{eqnarray}
where $H, G$ are measurable functions and $W_k, V_{k}$ are mutually independent i.i.d. noise processes (taking values without loss in $[0,1]$) and $X_0 \sim \mu$.

A POMP (Partially Observable Markov Process) can be reduced to a fully observable Markov process by defining the filter process \cite{Rhe74, Yus76}:
$$
Z_n:=\operatorname{Pr}\left\{X_n \in \cdot \mid Y_0, \ldots, Y_n\right\}.
$$
with state space $\mathcal{Z}$. 
The transition kernel $\eta: \mathcal{Z} \rightarrow \mathbb{P}(\mathcal{Z})$  is defined as:
\begin{align}\label{etatra}
\eta(\cdot \mid z)&=\int_{\mathbb{Y}} \1_{\{F(z, y_{n+1}) \in \cdot\}}\operatorname{Pr}(y_{n+1}\in dy \mid z_n=z)\\
& =   \int \1_{\{F(z, y_{n+1}) \in \cdot\}}\Q(dy_{n+1}|x_{n+1}){\T}(dx_{n+1}|x_n)z(dx_n), \nonumber
\end{align}
where
\begin{equation}\label{F_def} 
    F(z, y)(\cdot):=\operatorname{Pr}\left\{X_{n+1} \in \cdot \mid Z_n=z, Y_{n+1}=y\right\}.
    \end{equation}

Hence, the filter process is a completely observable Markov process with the state space $\mathcal{Z}$ and transtion kernel $\eta$.

\begin{definition}\label{Dobrushincoefficient}[\cite{dobrushin1956central}, Equation 1.16.] For a kernel operator $K: S_1 \rightarrow \mathcal{P}\left(S_2\right)$, 
    the Dobrushin coefficient is defined by:
    \begin{equation}
\delta(K)=\inf \sum_{i=1}^n \min \{ K\left(x, A_i\right), K\left(y, A_i\right)\}
\end{equation}
where the infimum is over all $x, y \in S_1$ and all finite partitions $\left\{A_i\right\}_{i=1}^n$ of $S_2$.
\end{definition}

A sequence $\left\{\mu_{n}\right\}_{n\in\mathbb{N}} \subset \mathcal{Z}$ converges weakly to $\mu \in \mathcal{Z}$ if for every $f \in \mathbb{C}(\mathbb{X})$,
\begin{equation}
\int_{\mathbb{X}} f(x) \mu_{n}(d x) \rightarrow \int_{\mathbb{X}} f(x) \mu(d x) \quad \text { as } \quad n \rightarrow \infty .
\end{equation}
We note that $\mathcal{Z}$ is a Polish space under the weak convergence topology when
 $(\mathbb{X}, d)$ is a complete, separable and metric space \cite[ Chapter 2, section 6]{Par67}. 
One way to metrize weak convergence on $\mathcal{Z}$ is via the bounded Lipschitz (BL) metric \cite[p.109]{villani2008optimal}:
\begin{align}\label{boundedlip}
&\rho_{BL}(\mu, \nu) :=\sup \left\{\int_{\mathbb{X}} f(x) \mu(d x)-\int_{\mathbb{X}} f(x) \nu(d x) : f \in \operatorname{BL_1}(\mathbb{X}) \right\},
\end{align}
where
\begin{align*}
\operatorname{BL}_1(\mathbb{X}):=\left\{f:
\norm{f}_{\infty}+\norm{f}_{L}\leq 1 \right\}, \; \norm{f}_{\infty}=\sup_{x\in \mathbb{X}}|f(x)|,
\; \norm{f}_{L}=\sup_{x\neq y} \frac{|f(x)-f(y)|}{d(x,y)}.
\end{align*}
By, e.g., \cite{Bil99}, if $(\mathbb{X}, d)$ is compact, then
$\PP(\mathbb{X})$, equipped with $\rho_{BL}$, 
is also compact.

When $\mathbb{X}$ is compact, the Kantorovich-Rubinstein metric (equivalent to the Wasserstein metric of order $1$) (\cite{Bog07}, Theorem 8.3.2) provides another way to metrize the weak topology on $\mathcal{Z}$ \cite[Theorem 8.3.2]{Bog07}. It is defined as:
\begin{align}\label{defkappanorm}
&W_1(\mu, \nu):=\sup \left\{\int_{\mathbb{X}} f(x) \mu(d x)-\int_{\mathbb{X}} f(x) \nu(d x) :\norm{f}_{L}\leq 1\right\}.
\end{align}

\subsection{Literature Review}

\yed{The unique ergodicity problem of non-linear filter processes has received significant interest at least since the seminal paper by Blackwell \cite{blackwell1959entropy}. Kaijser \cite{Kaijser} studies unique ergodicity in the context of filtering processes for 
finite state and finite observation spaces, and this study involves the use of the Furstenberg-Kesten theory of random matrix products. Kochman and Reeds prove a more general result in \cite{kochman2006simple}, also for the finite space case; and assume that the closure of finite multiplications of transition matrices has a rank one element. 
A further study by Kaijser \cite{kaijser2011markov} extends the results of Kochman and Reeds to countable $\mathbb{X}$. Chigansky and Van Handel \cite{chigansky2010complete} obtain necessary and sufficient conditions for unique ergodicity of the filter 
in the case where $\mathbb{X}$ and $\mathbb{Y}$ both take values in a countable state space. Additionally, the condition presented in \cite{chigansky2010complete} is also sufficient when the observation space is $\mathbb{R}^d$. \cite{chigansky2010complete} notes also that when $\mathbb{X}$ is finite and $\mathbb{Y}$ is countable, the condition in \cite{kochman2006simple} is necessary and sufficient for unique ergodicity of the filter process. 

In particular, the equivalence of unique ergodicity and almost sure filter stability has been a powerful approach to establish unique ergodicity (see e.g. \cite[Theorem 3.1]{chigansky2010complete}): If $X_n$ is an ergodic process, then almost sure filter stability implies unique ergodicity under several conditions in both discrete and continuous-time settings 
\cite{Budhiraja,DiMasiStettner2005ergodicity,chigansky2010complete,Handel,le2004stability}. 

Since the filter process is Markovian with kernel (\ref{etatra}), another approach to unique ergodicity could be to view the filter as a measure valued Markov process, and develop the necessary tools and conditions. To this end, there are two challenges: (i) Unique ergodicity via irreducibility properties \cite{MeynBook} of Markov chains is not applicable to filter processes in general. (ii) Unique ergodicity via the existence of a reachable state together with the strong Feller property \cite{Hairer,da1996ergodicity} (which is generally not applicable to filter processes; see \cite{FeKaZg14,KSYWeakFellerSysCont}), or via equi-continuity \cite{MeynBook,komorowski2010ergodicity} (whose applicability has been an open problem for filter processes) is a more promising method. To this end, however, while weak Feller continuity properties of filters have been recently studied \cite{CrDo02,FeKaZg14,KSYWeakFellerSysCont,kara2020near}, stronger regularity is needed for such a reachability approach to be consequential. 

In our paper we develop the machinery to facilitate (ii) above. Specifically, we establish conditions ensuring uniform weak continuity of transition kernels across time stages which then imply unique ergodicity provided that a topologically reachable state exists, and then present explicit conditions that guarantee the existence of a topologically reachable filter state.}
 
\section{Main Results}

\subsection{Unique Ergodicity.}
We begin with continuity assumptions for the transition and observation kernels:
\begin{assumption}\label{weakQtvT}
    There exists $\alpha\in \mathbb{R}^+$ such that 
    $\left\|\mathcal{T}(\cdot | x)-\mathcal{T}\left(\cdot | x^{\prime}\right)\right\|_{T V} \leq \alpha \cdot d(x,x^{\prime})$
    for every $x,x' \in \mathbb{X}$. 
\end{assumption}
\begin{assumption}\label{weakTtvQ}
\begin{itemize}
    \item[(i)] There exists a constant $\theta \in (0,1)$ such that
$
W_1\left(\mathcal{T}(\cdot | x), \mathcal{T}\left(\cdot | x^{\prime}\right)\right) \leq \theta \cdot d\left(x, x^{\prime}\right)
$
for every $x, x^{\prime} \in \mathbb{X}$.
\item [(ii)] There exists a constant $\gamma \in \mathbb{R}^{+}$ such that
$
\left\|\mathcal{Q}(\cdot \mid x) - \mathcal{Q}\left(\cdot \mid x^{\prime}\right)\right\|_{TV} \leq \gamma \cdot d\left(x, x^{\prime}\right)
$
for every $x, x^{\prime} \in \mathbb{X}$.
\end{itemize}
\end{assumption}
Under either Assumption~\ref{weakQtvT} or~\ref{weakTtvQ}, the filtering kernel $\eta$ defined in~\eqref{etatra} satisfies the weak Feller property \cite{KSYWeakFellerSysCont}.
\begin{definition}
    Let $P$ be a transition kernel.
    A sate $x\in X$ is  
    (i) \emph{topologically reachable} if 
        for every y and every open neighborhood 
        $O$ of $x$, there exists $k\in \mathbb{N}$ such that $P^k(y, O)>0$,
    (ii) \emph{topologically aperiodic} 
    if  for every open neighborhood 
    $O$ of $x$ there exists $K\in \mathbb{N}$ such that 
     $P^k(x, O)>0$ for all $k \geq K$.   
\end{definition}

Our main result on unique ergodicity is as follows:
\begin{theorem}[\bf{Unique Ergodicity under Reachability}]\label{Main2} 
    Under either Assumption~\ref{weakQtvT} or~\ref{weakTtvQ},
    if the filter process has a topologically reachable state, then
    it is uniquely ergodic. 
    Moreover, if the reachable state
    is aperiodic, then for any initial distribution,
    the filter process
    converges weakly to the unique invariant measure.
    \end{theorem}
    For Markov chains which are not irreducible, 
    topologically reachable states play a crucial role, 
    not unlike small or petite sets \cite{MeynBook}  serving as recurrent 
    sets for the irreducible setting. Our analysis, 
    in Section \ref{existence_reachable}, establishes the existence 
    of a reachable state and its implications for unique 
    ergodicity and weak convergence.
    \begin{corollary}\label{Main2Cor}
        If $\mathbb{X}$ is finite and the filter process has a topologically reachable state, then it is uniquely ergodic.
    \end{corollary}

\subsection{Existence of a reachable state for the filter process}
\label{existence_reachable}
This section presents conditions guaranteeing  the existence of a reachable state for the filter process. 

\subsubsection{Conditions via Contraction Along a Fixed Measurement}
First, we present a condition on mixing properties of the transition kernel. 
\begin{definition}
    Two non-negative measures $\mu, \mu^{\prime}$ on $(\mathbb{X},\B(\mathbb{X}))$ are comparable, if there exist positive constants $0<a \leq b$, such that
    $
    a \mu^{\prime}(A) \leq \mu(A) \leq b \mu^{\prime}(A)
    $
    for any Borel subset $A \subset \mathbb{X}$.
\end{definition}

\begin{definition}[Mixing kernel]
    The non-negative kernel $K$ defined on $\mathbb{X}$ is mixing, if there exists a constant $0<\varepsilon \leq 1$, and a non-negative measure $\lambda$ on $\mathbb{X}$, such that
    $
    \varepsilon \lambda(A) \leq K(x, A) \leq \frac{1}{\varepsilon} \lambda(A)
    $
    for any $x \in \mathbb{X}$, and any Borel subset $A \subset \mathbb{X}$.
    \end{definition}

\begin{definition}(Hilbert metric).  Let $\mu, \nu$ be two non-negative finite measures. We define the Hilbert metric on such measures as
    \begin{equation}
    h(\mu, \nu)= \begin{cases}\log \left(\frac{\sup _{A \mid \nu(A)>0} \frac{\mu(A)}{\nu(A)}}{\inf _{A \mid \nu(A)>0} \frac{\mu(A)}{\nu(A)}}\right) & \text { if } \mu, \nu \text { are comparable } \\ 0 & \text { if } \mu=\nu=0 \\ \infty & \text { else }\end{cases}
    \end{equation}
\end{definition}

Note that $h(a\mu, b\nu) = h(\mu, \nu)$ for any positive scalars $a, b$. Therefore, the Hilbert metric is a useful metric for nonlinear filters, and the following lemma demonstrates that it bounds the total-variation distance.

\begin{lemma}[\cite{le2004stability}, Lemma 3.4.]\label{h-TV}
    Let $\mu, \mu^{\prime}$ be two non-negative finite measures, and let 
$\bar{\mu}$ and $\bar{\mu}^{\prime}$ denote their normalized versions (i.e., probability measures obtained by normalization).  
    \begin{enumerate}
    \item[i.] $
    \bigl\|\bar{\mu} - \bar{\mu}^{\prime}\bigr\|_{TV} \;\leq\; \frac{2}{\log 3}\, h(\mu,\mu^{\prime}).
$
    \item[ii.] If the nonnegative kernel $K$ is mixing kernel with constant $\epsilon$, then 
    $
  h(K\mu, K\mu^{\prime}) \;\leq\; \frac{1}{\varepsilon^2}\, 
\bigl\|\bar{\mu} - \bar{\mu}^{\prime}\bigr\|_{TV}.
  $
    \end{enumerate}
\end{lemma}

\begin{assumption} \label{mixing_kernel_con} 
    \begin{enumerate}
    \item[i.] $\mathbb{Y}$ is countable and 
    there exists an observation realization 
    $y' \in \mathbb{Y}$ such that for some 
    positive $\epsilon$, $\Q(y'|x)>\epsilon$ 
    for every $x\in \mathbb{X}$. 
    \item[ii.] The transition kernel ${\T}$ is mixing kernel.
   \end{enumerate}
\end{assumption}

\begin{lemma}[\cite{le2004stability}, Lemma 3.8]\label{Birkoff} 
    (Birkhoff contraction coefficient). 
    The nonnegative linear operator on $\mathcal{M}^{+}(\mathbb{X})$ (positive measures on $\mathbb{X}$) 
    associated with a nonnegative kernel $K$ defined on $\mathbb{X}$, is a contraction under the Hilbert metric, and
    $$
    \tau(K):=\sup _{0<h\left(\mu, \mu^{\prime}\right)<\infty} \frac{h\left(K \mu, K \mu^{\prime}\right)}{h\left(\mu, \mu^{\prime}\right)}=\tanh \left[\frac{1}{4} H(K)\right]
    $$
    where
    $$
    H(K):=\sup _{\mu, \mu^{\prime}} h\left(K \mu, K \mu^{\prime}\right)
    $$
    is over nonnegative measures. $\tau(K)$ is called the Birkhoff contraction coefficient.
    Notice that $H(K)<\infty$ implies $\tau(K)<1$.
\end{lemma}

Recall from equation (\ref{F_def}) that $
F(z, y)(\cdot)=\operatorname{Pr}\left\{X_{n+1} \in \cdot \mid Z_n=z, Y_{n+1}=y\right\}
$.
        \begin{lemma}\label{clm}
    Under Assumption \ref{mixing_kernel_con}, 
    there exists a constant $c<1$  such that 
    \begin{align}
        h(F(\mu, y'), F(\nu, y'))\leq c h(\mu, \nu)
    \end{align}
    for every comparable $\mu,\nu\in \Z$.
    \end{lemma}
    \begin{proof}
    For the fixed observation $y'$ 
    given in Assumption \ref{h-TV},
    define the nonnegative kernel 
    $K: \mathcal{M}^{+}(\mathbb{X}) \to \mathcal{M}^{+}(\mathbb{X})$ 
    by
    $$
    K \rho\left(d x^{\prime}\right)=\int_\mathbb{X}\rho(dx) \Q(y'|x')\T(dx'|x),
    $$
    for any nonnegative measure $\rho \in \mathcal{M}^{+}(\mathbb{X})$.
    Define $$N^\mu_{y'}:=\int_\X\int_\X \Q(y'|\Bar{x})\T(d\Bar{x}|x)\mu(dx).$$
    $N^\mu_{y'}, N^\nu_{y'}>0$ because of Assumption \ref{mixing_kernel_con}. 
    
    \begin{align*}
    &P^{\mu}\left(X_1\in A \mid Y_{1}=y'\right)=\frac{\int_\X\int_A \Q(y'|x')\T(dx'|x)\mu(dx)}{N^\mu_{y'}}=\frac{(K\mu)(A)}{N^\mu_{y'}}.
    \end{align*}
    Therefore,
    \begin{align}
        h(F(\mu, y'), F(\nu, y'))=h(P^{\mu}\left(X_1\in \cdot \mid Y_{1}=y'\right),P^{\nu}\left(X_1\in \cdot \mid Y_{1}=y'\right))=h(K\mu, K\nu)
    \end{align}
    By Lemma \ref{Birkoff}
    $$
    \tau(K)=\sup _{0<h\left(\mu, \nu\right)<\infty} \frac{h\left(K \mu, K \nu\right)}{h\left(\mu, \nu\right)}<1
    $$
    only if $H(K)=\sup _{\mu, \mu^{\prime}} h\left(K \mu, K\mu^{\prime}\right)<\infty$. 
    Given that $\Q(y'|x) > \epsilon$ for every $x \in \mathbb{X}$, it follows that
    \begin{align}\label{K_T}
    \int_{\X} {\T}(A|x)\mu(dx)\geq (K\mu)(A) \geq \epsilon \int_{\X}  {\T}(A|x)\mu(dx)
    \end{align}
    for every $A\in \B(\mathbb{X})$.
    
    Let $\epsilon_{\T}$ be constant of the
    mixing kernel $\T$, i.e., 
    there exists a non-negative measure $\lambda$ on $\mathbb{X}$, such that
    $$
    \epsilon_{\T} \lambda(A) \leq \T(A|x) \leq \frac{1}{\epsilon_{\T}} \lambda(A)
    $$
    for any $x \in \mathbb{X}$, and any Borel subset $A \subset \mathbb{X}$.
    Thus, from equation (\ref{K_T}), we obtain:
    \begin{align}
        \frac{\mu(\X)}{\epsilon_{\T}} \lambda(A) \geq (K\mu)(A) \geq \epsilon \mu(\X) \epsilon_{\T}  \lambda(A)
    \end{align}
    for every $A\in \B(\mathbb{X})$.

    Hence, for any two finite positive measures 
$\mu, \nu$, we obtain:
    \begin{align}\label{mu_nu}
        \frac{\mu(\X)}{\nu(\X)\epsilon_{\T}^2 \epsilon} 
        \geq \frac{(K\mu)(A)}{(K\nu)(A)} \geq 
        \frac{\mu(\X)\epsilon_{\T}^2 \epsilon}{\nu(\X)}.
    \end{align}

    Lemma \ref{Birkoff} and equation (\ref{mu_nu})
    conclude the proof:
\begin{align*}
    \tau(K)&=\sup _{0<h\left(\mu, \nu\right)<\infty} \frac{h\left(K \mu, K \nu\right)}{h\left(\mu, \nu\right)}
    =\tanh \left[\frac{1}{4} H(K)\right]
    =\tanh \left[\frac{1}{4} \sup _{\mu, \nu} h\left(K \mu, K \nu\right)\right]
    \\ & \leq \tanh \left[\frac{1}{4} \log \left( \frac{1}{\epsilon_{\T}^4 \epsilon^2} \right)\right]
     =\frac{1-\epsilon_{\T}^2 \epsilon}{1+\epsilon_{\T}^2 \epsilon}<1.
\end{align*}
\end{proof}
\begin{lemma}\label{Cauchy}
    Under Assumption \ref{mixing_kernel_con},
    there exists an aperiodic reachable state 
    for filter process.
\end{lemma}

\begin{proof}
Define 
$\pi_n^{\mu *}(\cdot)=P^{\mu}\left(X_n \in \cdot \mid Y_{[0, n]}=(y',\dots, y')\right), n \in \mathbb{N}$ 
where $P^{\mu}$ is the probability measure induced by the prior $\mu$.
First we prove that $\pi_n^{\mu *}$ is 
a Cauchy sequence on $\Z$ under total-variation metric.
Under Assumption~\ref{mixing_kernel_con}, the conditional measures 
$\pi_{n+1}^{\mu *}$ and $\pi_{n}^{\mu *}$ are comparable for every $n \geq 1$.  
Indeed, Assumption~\ref{mixing_kernel_con}(ii) states that the transition kernel ${\T}$ is a mixing kernel; that is, there exist $\varepsilon_\T > 0$ and a nonnegative measure $\lambda$ such that 
\(
    \varepsilon_\T \lambda(A) \;\leq\; \T(A\mid x) \;\leq\; \frac{1}{\varepsilon_\T}\lambda(A),
    \;\; \forall x \in \mathbb{X}, \;\; \forall A \in \mathcal{B}(\mathbb{X}).
\)
Furthermore, Assumption~\ref{mixing_kernel_con}(i) guarantees that the likelihood of observing $y'$ is uniformly bounded away from zero across all states, i.e., $\Q(y' \mid x) \geq \epsilon > 0$ for every $x \in \mathbb{X}$.  

Combining these two properties, we obtain for every $n \geq 1$,
\[
   \varepsilon_\T \epsilon \,\lambda(A) \;\leq\; \pi_n^{\mu *}(A) \;\leq\; \frac{1}{\varepsilon_\T \epsilon}\,\lambda(A),
   \qquad \forall A \in \mathcal{B}(\mathbb{X}).
\]
Hence, $\pi_{n+1}^{\mu *}$ and $\pi_n^{\mu *}$ are mutually absolutely continuous, with Radon–Nikodym derivatives bounded above and below by positive constants depending only on $\varepsilon_\T$ and $\epsilon$. In other words, they are comparable.  

Therefore, by Lemma~\ref{clm}, there exists some $c<1$ such that
\[
h\!\left(\pi_{n+2}^{\mu *}, \pi_{n+1}^{\mu *}\right) \;\leq\; c \, h\!\left(\pi_{n+1}^{\mu *}, \pi_{n}^{\mu *}\right).
\]
Thus, the sequence is a Cauchy sequence under the Hilbert metric, and consequently under the total-variation metric.
Since $\Z$ is a complete metric space, 
there exists a probability measure $\pi^{\mu *}\in \mathcal{Z}$ which is the limit of the Cauchy sequence $(\pi_n^{\mu *})_{n\in N}$. 
Now we will show that the limit is
independent of the initial measure $\mu$.
For any two initial measures 
$\mu$ and $\nu$, 
$F(\mu, y')$ and $F(\nu, y')$ are comparable
since $K$ is a mixing kernel. Therefore, 
$
h\left(\pi_{n+1}^{\mu *}, \pi_{n+1}^{\nu *}\right) \leq c h\left(\pi_{n}^{\mu *}, \pi_{n}^{\nu *}\right)
$ for $n\geq 1$. Consequently, the limits of 
both sequences are the same. Thus, 
the limit point is independent of the initial measure. 
Let $\pi^*$ be the limit point.

For any open neighborhood $O$ of $\pi^{*}$, 
there exist $\kappa>0$
such that $O \supset B_\kappa(\pi^{*}):=
\{\mu\in \mathcal{Z}: \rho_{BL}(\mu,\pi^{*})<\kappa\}$.
Under Assumption \ref{mixing_kernel_con},
we know that $P^{\mu}(Y_1=y')>\epsilon$ 
for any $\mu$.
Since convergence in total variation implies
convergence in the bounded Lipschitz metric,
there exists a sufficiently large $M$ such that,
for any $n>M$,
$\pi_n^{\mu *}(\cdot)\in O$ and 
$P^{\mu}(Y_{[0, n]}=(y',\dots, y'))>0$
This is true for any initial prior $\mu$,
so $\pi^{*}$ is topologically reachable.
It is also true for $\mu=\pi^{*}$, 
meaning $\pi^{*}$ is topologically aperiodic.
Thus, $\pi ^*$ is a topologically 
reachable aperiodic state. 
\end{proof}

As a corollary, we state the following.

\begin{theorem}[\bf{Unique Ergodicity of the Filter Process}]\label{Main21} 
    Under either Assumption~\ref{weakQtvT} or~\ref{weakTtvQ}, and Assumption \ref{mixing_kernel_con}, the filter process is uniquely ergodic.
    Moreover, for any initial distribution,
    the filter process
    converges weakly to the unique invariant measure.
\end{theorem}

\yed{\begin{remark}
In the analysis above, we do not impose filter stability, and accordingly results that rely on the equivalence of filter stability and unique ergodicity under certain technical conditions (e.g.,  \cite[Theorem~3.1]{chigansky2010complete}) do not apply to our special case. Notably, \cite[Theorem~3.1]{chigansky2010complete} establishes equivalence between almost sure filter stability and unique ergodicity under conditions where the stability requirement is stronger than what is possible under our conditions (e.g. the one obtained in Theorem~2.1 of \cite{chigansky2009intrinsic}, which assumes uniform absolute continuity with bounds on the Radon–Nikodym derivatives between initial conditions). We also note that, Hilbert metric-based filter stability results (e.g., \cite{le2004stability}) require strong mixing properties for both the transition kernel ${\T}$ and the observation kernel $\Q$, rather than only on ${\T}$. By contrast, our conditions here require contraction only under specific measurement realizations for a fixed observation $\bar{y}$, repeated sequentially, rather than uniform contraction across all observations. Thus, we do not impose general filter stability, but only a contraction condition when a particular measurement is realized. Consequently, Assumption \ref{mixing_kernel_con} alone does not suffice to ensure the mixing kernel condition in \cite{le2004stability}.  
\end{remark}}

\subsubsection{Conditions under two Cases of Measurement Realizations without Transition Kernel Mixing: Informative or Non-Informative Measurements}

For some applications, the existence of a reachable state follows from more direct arguments for two special cases: (i) when a non-informative measurement takes place with positive probability, or (i) when there exists a state which is uniquely recoverable from a measurement. We start with the former:

\begin{assumption} \label{positive_eq} 
    The transition kernel ${\T}$ is uniformly ergodic with
    invariant probability measure $\pi^{*}$; that is, 
    \[
\sup_{\mu} \big\| P^{\mu}(X_n \in \cdot) - \pi^{*} \big\|_{TV} \;\xrightarrow[n \to \infty]{}\; 0,
\]
where $P^{\mu}$ denotes the law of the Markov chain with initial distribution $\mu$.
 Furthermore, there exists an observation realization 
    $\bar{y} \in \mathbb{Y}$ such that for some 
    positive $\epsilon$, $\Q(\bar{y}|x)=\epsilon$ 
    for every $x\in \mathbb{X}$. 
\end{assumption}

\begin{definition}
    \begin{align}
        \pi_n^{\mu *}(\cdot)=P^{\mu}\left(X_n \in \cdot \mid Y_{[0, n]}=(\bar{y},\dots, \bar{y})\right), n \in \mathbb{N}
    \end{align} where $P^{\mu}$ is the probability measure induced by the prior $\mu$.
\end{definition}

\begin{lemma}\label{reachability}
Under Assumption \ref{positive_eq}, 
$\pi_n^{\mu *}(\cdot)$ converges to the unique 
invariant probability
measure $\pi^{*}$
for any initial measure $\mu \in \Z$ under total variation. 
Furthermore,
the limit measure $\pi^{*}$ is a 
topologically aperiodic reachable state for the filter process.
\end{lemma}
\begin{proof}
\begin{align*}
    & \pi_n^{\mu *}(A)=P^{\mu}\left(X_n \in A \mid Y_{[0, n]}=(\bar{y},\dots, \bar{y})\right)
    \nonumber \\ & = \frac{P^{\mu}\left(Y_n = \bar{y} \mid X_n \in A,  Y_{[0, n-1]}=(\bar{y},\dots, \bar{y})\right) P^{\mu}\left(X_n \in A \mid Y_{[0, n-1]}=(\bar{y},\dots, \bar{y})\right)}
    {P^{\mu}\left(Y_n = \bar{y} \mid Y_{[0, n-1]}=(\bar{y},\dots, \bar{y})\right)}
    \nonumber \\ & =\frac{P^{\mu}\left(Y_n = \bar{y} \mid X_n \in A,  Y_{[0, n-1]}=(\bar{y},\dots, \bar{y})\right) P^{\mu}\left(X_n \in A \mid Y_{[0, n-1]}=(\bar{y},\dots, \bar{y})\right)}
    {\int_{x_n \in \mathbb{X}} P^{\mu}\left(Y_n = \bar{y} \mid X_n=x_n,  Y_{[0, n-1]}=(\bar{y},\dots, \bar{y})\right)P^{\mu}\left(X_n \in d x_n \mid Y_{[0, n-1]}=(\bar{y},\dots, \bar{y})\right) }
    \nonumber \\ & = \frac{\epsilon P^{\mu}\left(X_n \in A \mid Y_{[0, n-1]}=(\bar{y},\dots, \bar{y})\right)}
    {\int_{x_n \in \mathbb{X}} \epsilon P^{\mu}\left(X_n \in d x_n \mid Y_{[0, n-1]}=(\bar{y},\dots, \bar{y})\right)}
    = P^{\mu}\left(X_n \in A \mid Y_{[0, n-1]}=(\bar{y},\dots, \bar{y})\right).
\end{align*}
The second line results from the Bayesian update equation. 
The last line derives from the fact that the current 
observation is independent of past observations given 
the current state, and the fact that 
under Assumption \ref{positive_eq}, 
$\Q(\bar{y}|x) = \epsilon$ for every $x \in \mathbb{X}$.

Then, the recursive application of the update rule leads to
\begin{align*}
    \pi_n^{\mu *}(A)=P^{\mu}\left(X_n \in A \mid Y_{[0, n]}=(\bar{y},\dots, \bar{y})\right)
    =P^{\mu}\left(X_n \in A \right).
\end{align*}
which by the uniform ergodicity of the transition kernel 
${\T}$, implies that $\pi_n^{\mu *}(\cdot)$ converges 
to $\pi^*$ under the total variation.
$\pi^{ *}$ is an aperiodic reachable state, as can 
be seen from the same argument presented 
in the final part of Lemma \ref{Cauchy}.
\end{proof}

\begin{corollary}\label{Main22} 
Under either Assumption~\ref{weakQtvT} or~\ref{weakTtvQ}, together with Assumption~\ref{positive_eq}, the filter process is uniquely ergodic.
Moreover, for any initial distribution, the filter process
converges weakly to the unique invariant measure.
\end{corollary}

\medskip

We conclude this section with another case where there exists a hidden state that is recoverable through a unique measurement, then the filter admits a reachable state.

\begin{assumption}\label{fully_recoverable_state} 
Assume that $\mathbb{X}$ is countable.  
The state $x \in \mathbb{X}$ is a reachable state for the transition kernel ${\T}$ in the sense that for every $z \in \mathbb{X}$ there exists $k \in \mathbb{Z}_+$ such that $P(X_k = x | X_0=z) > 0$. Furthermore, there exists an observation realization 
$\bar{y} \in \mathbb{Y}$ such that 
$\Q(\bar{y}\mid x)=1$ and $\Q(\bar{y}\mid x')=0$ for every $x'\neq x$.  
That is, the reachable state $x$ is fully recoverable.
\end{assumption}

\begin{lemma}\label{reachability2}
Under Assumption~\ref{fully_recoverable_state}, 
the Dirac measure $\delta_x$ is a 
topologically reachable state for the filter process.  
Furthermore, if $x$ is aperiodic for ${\T}$, then $\delta_x$ is a
topologically aperiodic reachable state for the filter process.
\end{lemma}

\begin{proof}
Under Assumption~\ref{fully_recoverable_state}, regardless of the initial distribution, the process visits state $x$ with positive probability since $x$ is reachable. Whenever the process enters $x$, the observation $\bar{y}$ occurs with probability $1$, so the filter state becomes $\delta_x$. Hence, $\delta_x$ is a topologically reachable state.  
If $x$ is also aperiodic for ${\T}$, then by the same argument, $\delta_x$ is a topologically aperiodic reachable state for the filter process.
\end{proof}

\begin{corollary}\label{Main23} 
Under either Assumption~\ref{weakQtvT} or~\ref{weakTtvQ}, together with Assumption~\ref{fully_recoverable_state}, the filter process is uniquely ergodic.
Moreover, if $x$ is aperiodic for ${\T}$, then for any initial distribution the filter process
converges weakly to the unique invariant measure.
\end{corollary}

\section{Proof of the Main Theorem}


\begin{lemma}\label{Regularity}
    Suppose $\mathbb{X}$ and $\mathbb{Y}$ are Polish spaces.
Then for all $n \in \mathbb{N}$ and $z,z' \in \mathcal{Z}$, 
\begin{itemize}
    \item[(i)] Under Assumption~\ref{weakQtvT}, we have
      $
    \rho_{B L}\left(\eta^n(\cdot \mid z), 
    \eta^n\left(\cdot \mid z^{\prime}\right)\right) \leq 
    2\left(1+\alpha \right) \rho_{B L}\left(z, z^{\prime}\right)
    $.
    \item[(ii)]Under Assumption \ref{weakTtvQ}, we have  $
    \rho_{B L}\left(\eta^n(\cdot \mid z), 
    \eta^n\left(\cdot \mid z^{\prime}\right)\right) \leq 
    2\left(2 +  \frac{\gamma \theta}{1-\theta} \right) \rho_{B L}\left(z, z^{\prime}\right).
    $ 
\end{itemize}
    
\end{lemma}

\begin{proof}[Proof of Lemma \ref{Regularity}]
    We adapt the argument from Kara and Y\"{u}ksel (\cite{kara2020near}, Theorem 7).
    We equip $\mathcal{Z}$
    with the metric $\rho_{BL}$ to define the bounded-Lipschitz 
    norm $|f|_{BL}$ for any Borel measurable function 
    $f: \mathcal{Z}\rightarrow \mathbb{R}$. 
    We treat the full measurement sequence $y_1^n := (y_1, y_2, \ldots, y_n)$  as a single observation variable to derive bounds uniformly in $n$.

    Define, for $z, z' \in \mathcal{Z}$ and $r\in R^+$:
    \begin{align}\label{rhodef}
        \rho_n(z,z',r) := 
        \sup_{f} \bigg| \int_{} f(y_1^n, x_n) \, \Pr(dy_1^n, dx_n | X_0 \sim z) 
         - \int f(y_1^n, x_n) \, \Pr(dy_1^n, dx_n | X_0 \sim z') \bigg|
    \end{align}
    where the supremum is over all measurable \( f : \mathbb{Y}^n \times \mathbb{X} \to \mathbb{R} \) such that \( \|f\|_{\infty} \leq 1 \),
     and  \( x_n \mapsto f(y_1^n, x_n) \) is $r$-Lipschitz for each fixed $y_1^n$.
     By definition, the following hold:
        \begin{align}
        &\left\| \Pr(Y_1^n \in \cdot \mid X_0 \sim z) - \Pr(Y_1^n \in \cdot \mid X_0 \sim z') \right\|_{TV} \leq \rho_n(z, z',r) , \label{1}\\
        &  \rho_n(z, z',r)  \leq \left\| \Pr((Y_1^n, X_n) \in \cdot \mid X_0 \sim z) - \Pr((Y_1^n, X_n) \in \cdot \mid X_0 \sim z') \label{2} \right\|_{TV}.
        \end{align}
        We now analyze the bounded-Lipschitz distance between the filtered distributions at time $n$, starting from two different initial beliefs $z$ and $z^{\prime}$ :
\begin{align} \label{secondlemma}
& \rho_{B L}\left(\eta^n(\cdot \mid z), \eta^n\left(\cdot \mid z^{\prime}\right)\right)\nonumber\\
& =\sup _{f\in \operatorname{BL}_1(\mathcal{Z})}\left|\int_{\mathbb{Y}} f\left(z_n\left(z_0^{\prime}, y_1^n\right)\right) P\left(d y_1^n \mid z_0^{\prime}\right)-\int_{\mathbb{Y}} f\left(z_n\left(z_0,  y_1^n\right)\right) P\left(d y_1^n \mid z_0\right)\right|
\end{align}
where $z_n\left(z, y_1^n\right):=F\left(F\left(\ldots F\left(z, y_1\right), \ldots\right), y_n\right)$ denotes the recursive filter update based on the observation sequence 
$y_1^n$.
To bound the expression, we decompose it as follows:
\begin{align}
&\left|\int_{\mathbb{Y}} f\left(z_n\left(z_0^{\prime}, y_1^n\right)\right) P\left(d y_1^n \mid z_0^{\prime}\right)-\int_{\mathbb{Y}} f\left(z_n\left(z_0,  y_1^n\right)\right) P\left(d y_1^n \mid z_0\right)\right|\nonumber\\
&\leq \left|\int_{\mathbb{Y}} f\left(z_n\left(z_0^{\prime}, y_1^n\right)\right) P\left(d y_1^n \mid z_0^{\prime}\right)-\int_{\mathbb{Y}} f\left(z_n\left(z_0^{\prime}, y_1^n\right)\right) P\left(d y_1^n \mid z_0\right)\right| \nonumber\\
&\quad + \int_{\mathbb{Y}}\left|f\left(z_n\left(z_0^{\prime}, y_1^n\right)\right)-f\left(z_n\left(z_0, y_1^n\right)\right)\right| P\left(d y_1^n \mid z_0\right) \nonumber\\
&\leq \norm{f}_{\infty} \left\|P\left(Y_1^n \mid z_0^{\prime} \right)-P\left(Y_1^n \mid z_0\right)\right\|_{T V}\nonumber\\
& \quad + \int_{\mathbb{Y}}\left|f\left(z_n\left(z_0^{\prime}, y_1^n\right)\right)-f\left(z_n\left(z_0, y_1^n\right)\right)\right| P\left(d y_1^n \mid z_0\right).\label{TVtv}
\end{align}
where the first term is bounded by $\norm{f}_{\infty} \rho_n(z, z^{\prime}, 1)$, by inequality (\ref{1}).

Now consider the second term in \eqref{TVtv}. Since $f \in \mathrm{BL}_1(\Z)$, we use the Lipschitz continuity of $f$ :
\begin{align}
    &\int_{\mathbb{Y}^n}\left|f\left(z_n\left(z_0^{\prime}, y_1^n\right)\right)
    -f\left(z_n\left(z_0, y_1^n\right)\right)\right| Pr\left(d y_1^n \mid z_0\right)\nonumber\\
    &\leq \norm{f}_{L} \int_{\mathbb{Y}^n}\rho_{BL}(z_n\left(z_0^{\prime}, y_1^n\right), z_n\left(z_0, y_1^n\right)) P\left(d y_1^n \mid z_0\right)\nonumber\\
&=\norm{f}_{L} \int_{\mathbb{Y}^n} \sup_{g \in \operatorname{BL}_1(\mathbb{X})}\left(\int g(x_n)z_n\left(z_0^\prime, y_1^n \right)(dx_n) - \int g(x_n)z_n\left(z_0, y_1^n\right)(dx_n)\right)P\left(d y_1^n \mid z_0\right)\label{eq3}
\end{align}
If we look at the term inside
\begin{align}\label{inside}
&\sup_{g \in \operatorname{BL}_1(\mathbb{X})}\left(\int_{\mathbb{X}} g(x_n)z_n\left(z_0^\prime, y_1^n \right)(dx_n) - \int_{\mathbb{X}} g(x_n)z_n\left(z_0, y_1^n\right)(dx_n)\right)\\
&=\sup_{g \in \operatorname{BL}_1(\mathbb{X})}\left(\int_{\mathbb{X}} g(x_1)w_{\mathbf{y}}(dx_1)\right),\nonumber
\end{align}
where $\mathbf{y}=y_1^n$ and $w_{\mathbf{y}}=(z_n\left(z_0^\prime, y_1^n\right)- z_n\left(z_0, y_1^n\right))$ 
which is a signed measure on $\mathbb{X}$.
Now, since 
$\operatorname{BL}_1(\mathbb{X})$ is closed, uniformly bounded and equicontinuos with respect to the sup-norm topology, the Arzelà–Ascoli theorem implies that 
 $\operatorname{BL}_1(\mathbb{X})$ is compact. 
Since a continuous function on a compact set attains its supremum,
 the set
$$
A_{\mathbf{y}}:=\left\{h_{\mathbf{y}}(x)=\arg\sup_{g \in \operatorname{BL}_1(\mathbb{X})}\left(\int_{\mathbb{X}} g(x)w_{\mathbf{y}}(dx)\right):h_{\mathbf{y}}(x)\in \BL(\mathbb{X})\right\}
$$
is nonempty for every ${\mathbf{y}}\in \mathbb{Y}^n$. The integral is continuous under respect to sup-norm.
Then, $A_{\mathbf{y}}$ is closed in the supremum norm topology.
Moreover, since $\mathbb{Y}^n$ and $\operatorname{BL}_1(\mathbb{X})$ are Polish spaces, the set
$\Gamma:=\left\{(\mathbf{y}, h_\mathbf{y}): h_{\mathbf{y}} \in A_{\mathbf{y}}\right\}$
is Borel measurable. 
By the Measurable Selection Theorem 
\footnote{[\cite{himmelberg1976optimal}, Theorem 2][Kuratowski Ryll-Nardzewski Measurable Selection Theorem]
Let $\mathbb{X}, \mathbb{Y}$ be Polish spaces and $\Gamma=(x, \psi(x))$ where $\psi(x) \subset \mathbb{Y}$ be such that, $\psi(x)$ is closed for each $x \in \mathbb{X}$ and let $\Gamma$ be a Borel measurable set in $\mathbb{X} \times \mathbb{Y}$. Then, there exists at least one measurable function $f: \mathbb{X} \rightarrow \mathbb{Y}$ such that $\{(x, f(x)), x \in \mathbb{X}\} \subset \Gamma$.
} there exists a measurable function
$h:\mathbb{Y}^n\to \operatorname{BL}_1(\mathbb{X})$ 
such that $h({\mathbf{y}})\in A_{\mathbf{y}}$ for all ${\mathbf{y}} \in \mathbb{Y}^n$. 
We denote this selection by  $g_{\mathbf{y}}=h({\mathbf{y}})$
and use it to replace the supremum inside the integral in \eqref{eq3}. Thus:
\begin{align}
&\int_{\mathbb{Y}^n} \sup_{g \in \operatorname{BL}_1(\mathbb{X})}\left(\int_{\mathbb{X}} g(x_n)z_n\left(z_0^\prime, y_1^n \right)(dx_n) - \int_{\mathbb{X}} g(x_n)z_n\left(z_0, y_1^n\right)(dx_n)\right)P\left(d y_1^n \mid z_0\right)\nonumber\\
&=\int_{\mathbb{Y}^n} \left(\int_{\mathbb{X}} g_{{\mathbf{y}}}(x_n)z_n\left(z_0^\prime, {\mathbf{y}} \right)(dx_n) - \int_{\mathbb{X}} g_{\mathbf{y}}(x_n)z_n\left(z_0, {\mathbf{y}}\right)(dx_n)\right) P\left(d {\mathbf{y}} \mid z_0\right)\nonumber\\
&=\int_{\mathbb{Y}^n} \int_{\mathbb{X}} g_{\mathbf{y}}(x_n)z_n\left(z_0^\prime, {\mathbf{y}}\right)(dx_n) P(d {\mathbf{y}} | z_0)-
\int_{\mathbb{Y}^n}\int_{\mathbb{X}} g_{\mathbf{y}}(x_n)z_n\left(z_0^\prime, {\mathbf{y}}\right)(dx_n) P(d {\mathbf{y}} | z_0^\prime)\nonumber\\
&+\int_{\mathbb{Y}^n} \int_{\mathbb{X}} g_{\mathbf{y}}(x_n)z_n\left(z_0^\prime, {\mathbf{y}}\right)(dx_n) P(d {\mathbf{y}} \mid z_0^\prime)-
\int_{\mathbb{Y}^n}\int_{\mathbb{X}} g_{\mathbf{y}}(x_n)z_n\left(z_0, {\mathbf{y}}\right)(dx_n) P(d {\mathbf{y}} \mid z_0)\nonumber \\
& \leq 2 \cdot \rho_n(z, z',1) \label{sup}
\end{align}
by the inequality (\ref{1}) and the definition of (\ref{rhodef}).

\noindent Finally, using this bound along with \eqref{TVtv} and \eqref{sup}, we obtain:
\begin{align}\label{final}
\rho_{B L}\left(\eta^n(\cdot \mid z), 
\eta^n\left(\cdot \mid z^{\prime}\right)\right) \leq 
\sup_{f\in \operatorname{BL}_1(\mathcal{Z}) } \left(\norm{f}_{\infty} + 2 \norm{f}_{L}  \right) \rho_n(z, z',1)
\leq 2 \cdot \rho_n(z, z',1).
\end{align}
This shows that it suffices to uniformly bound $\rho_n\left(z, z^{\prime}, 1\right)$ to complete the proof. We now derive such a bound under the given assumptions.

(i) Since the process satisfies the Markov property  
$X_0 \rightarrow X_1 \rightarrow\left(Y_1^n, X_n\right)$, we can upper bound the total variation between the joint distributions of
$\left(Y_1^n, X_n\right)$ as follows:
$$ \rho_n(z, z',1)\leq 
\|\Pr(Y_1^n, X_n|z) - \Pr(Y_1^n, X_n|z^{\prime}) \|_{TV} \leq  
\|\Pr(X_1|X_0=z) - \Pr(X_1|X_0=z^{\prime}) \|_{TV}
$$
Then, using Assumption~\ref{weakQtvT} and the definition of the bounded-Lipschitz metric, we obtain
$$
\rho_n(z, z',1)\leq \|\Pr(X_1|X_0=z) - \Pr(X_1|X_0=z^{\prime}) \|_{TV} \leq (1+ \alpha) \cdot \rho_{BL}(z, z^{\prime}).
$$

(ii)We now establish a recursive upper bound for \( \rho_n(z, z', r) \), valid for any \( r \in \mathbb{R}^+ \), under Assumption~\ref{weakTtvQ}.
\begin{align*}
    & \rho_n(z, z', r) 
    = \sup_{f} 
    \left| \int f(y_1^n, x_n) \, \Pr(dy_1^n, dx_n |X_0 \sim z) 
    - \int f(y_1^n, x_n) \, \Pr(dy_1^n, dx_n |X_0 \sim z') \right| \\
    &= \sup_{f} 
    \left| \int h_f(y_1^{n-1}, x_{n-1}) \Pr(dy_1^{n-1}, dx_{n-1} | z) 
    - \int h_f(y_1^{n-1}, x_{n-1}) \Pr(dy_1^{n-1}, dx_{n-1} | z') \right|,
    \end{align*}
    where
    \(
    h_f(y_1^{n-1}, x_{n-1}) := \int f(y_1^n, x_n) \, Q(dy_n \mid x_n) \, \mathcal{T}(dx_n \mid x_{n-1}).
    \)
    Since \( \|f\|_\infty \leq 1 \), we also have \( \|h_f\|_\infty \leq 1 \). To control the Lipschitz constant of \( h_f \), fix \( y_1^{n-1} \) and consider:
    \begin{align*}
    & h_f(y_1^{n-1}, x_{n-1}) - h_f(y_1^{n-1}, x_{n-1}') =
     \int_{\mathbb{X}} w(x_n) \, \mathcal{T}(dx_n | x_{n-1}) 
    - \int_{\mathbb{X}} w(x_n) \, \mathcal{T}(dx_n | x_{n-1}'),
    \end{align*}
    where 
    \(
    w(x_n) := \int_{\mathbb{Y}} f(y_n, x_n) \, Q(dy_n \mid x_n).
    \)
    Using the Lipschitz property of \( f \) in the second argument and Assumption~\ref{weakTtvQ}-(ii):
    \begin{align*}
    &|w(x_n) - w(x_n')| 
    = \left| \int f(y_n, x_n) \, Q(dy_n \mid x_n) - \int f(y_n, x_n') \, Q(dy_n \mid x_n') \right| \\
    &\leq \int |f(y_n, x_n) - f(y_n, x_n')| \, Q(dy_n \mid x_n) 
    + \int f(y_n, x_n') \, |Q(dy_n \mid x_n) - Q(dy_n \mid x_n')| \\
    &\leq \norm{f}_L d(x_n, x_n') + \gamma \cdot d(x_n, x_n') = (r + \gamma) \cdot d(x_n, x_n'),
    \end{align*}
    Then, applying Assumption~\ref{weakTtvQ}-(i), we obtain:
    \(
    \|h_f\|_L \leq \theta \cdot \|w\|_L \leq (r + \gamma) \theta.
    \)
    Consequently, we obtain the recursive relation:
    \(
    \rho_n(z, z', r) \leq  \rho_{n-1}(z, z', (r + \gamma) \theta).
    \)
    Applying this recursively completes the proof:
    \[
    \rho_n(z, z',1) \leq \cdots \leq \rho_0(z, z',  \theta^n + \gamma \theta \frac{1-\theta^{n}}{1-\theta}) \leq \left(1+  \theta^n + \gamma \theta \frac{1-\theta^{n}}{1-\theta} \right) \cdot \rho_{BL}(z, z').
    \]
\end{proof}
We define the $L$-chain property as follows, which will be a critical notion for our analysis:
\begin{definition}
A Markov chain with transition kernel $P$ is called an $L$-chain if for each Lipschitz continuous function with compact support, the sequence of functions $\{\int P^n(d y\mid z) f(y), n \in \mathbb{N}\}$ are equi-continuous.
\end{definition}


\begin{theorem}\label{L-chain} Under either Assumption~\ref{weakQtvT} or~\ref{weakTtvQ},             
    the filter process is an L-chain. 
\end{theorem}
\begin{proof}
$\{\int \eta^n(d y | z) f(y), n \in \mathbb{N} \}$ is equi-continuous
for $f$ Lipschitz continuous function with compact support, because $f/(\norm{f}_L+\norm{f}_\infty) \in BL_1(Z)$ and 
result follows by Lemma \ref{Regularity}.
\end{proof}

An $L$-chain, by definition, satisfies the weak Feller property. 
Therefore, proving that a chain is an $L$-chain is sufficient 
to show that it also satisfies the weak Feller property. 
\begin{definition}
Let $\pi$ be a probability measure on $\mathbb{X}$ with metric $d$. The topological support of $\pi$ is defined with
$$
\operatorname{supp} \pi:=\left\{x: \pi\left(B_r(x)\right)>0\right\}, \quad \forall r>0,
$$
where $B_r(x)=\{y \in \mathbb{X}: d(x, y)<r\}$.
\end{definition}

\begin{theorem}[\cite{Hernandez-Lerma2003}, Theorem 7.2.3] \label{compactness}
    Let $\left\{x_t\right\}$ be a weak Feller Markov
    process taking values in a compact subset of a complete,
    separable metric space. Then $\left\{x_t\right\}$
    admits an invariant probability measure.
\end{theorem}

\begin{theorem}\label{Essential}
    Let a Markov chain be an \emph{L-chain} with compact state space $\mathcal{Z}$, and suppose there exists a topologically reachable state $z^* \in \mathcal{Z}$. 
   Then, (i) there exists a unique invariant probability measure, and (ii)
     if the reachable state $z^*$ is topologically aperiodic, the Markov process converges weakly to the unique invariant
     probability measure for every initial prior.
\end{theorem}

\begin{proof} By Theorem \ref{compactness} there is an invariant probability measure. Suppose that there were two different probability measures $\nu_1, \nu_2$. We may assume $\nu_1$ and $\nu_2$ to be ergodic, via an ergodic decomposition argument of invariant measures [\cite{Tweedie-94}, Theorem 6.1].
    Now, by equicontinuity and the Arzelà–Ascoli theorem, for every bounded and Lipschitz continuous function $f$, the sequence of functions   \begin{align}\label{subseq}
    \eta^{(N)}(f)(z):=\frac{1}{N} \sum_{k=0}^{N-1} \int \eta^k(d y|z) f(y)
    \end{align}
    has a convergent subsequence (in the sup norm) with a continuous limit $F_f^*: \mathcal{Z}\rightarrow\mathbb{R}$.
   
Since $z^*$ is a reachable state, 
for every neighborhood $B_r(z^*)$ with 
$r>0$, both invariant probability measures 
$\nu_1$ and $\nu_2$ assign positive 
probability to $B_r(z^*)$. Thus, $z^*$ 
belongs to the topological support of 
both $\nu_1$ and $\nu_2$. Then, there exist sequences ${x_n}$ and ${y_n}$ converging to $z^*$, 
for which the ergodic theorem holds. More explicitly, for any bounded and continuous function $f$, we have
$\lim_{N \rightarrow \infty} \eta^{(N)}(f)(y_n) = \langle \nu_1, f\rangle := \int_{\mathcal{Z}} f(z)\,\nu_1(dz)$,
$\lim_{N \rightarrow \infty} \eta^{(N)}(f)(x_n) = \langle \nu_2, f\rangle := \int_{\mathcal{Z}} f(z)\,\nu_2(dz)$.

The above imply that, every bounded and Lipschitz continuous function $f$, the term
$$
\lim _{n \rightarrow \infty}\left|\lim _{N \rightarrow \infty} \eta^{(N)}(f)\left(y_n\right)-\eta^{(N)}(f)\left(x_n\right)\right|=0 .
$$
Suppose not; there would be an $\epsilon>0$ and a subsequence $n_k$ for which the difference
$$
\left|\lim _{N \rightarrow \infty} \eta^{(N)}(f)\left(y_{n_k}\right)-\eta^{(N)}(f)\left(x_{n_k}\right)\right|>\epsilon .
$$
However, for each fixed $n_k$, we have that
$
\lim _{N \rightarrow \infty} \eta^{(N)}(f)\left(y_{n_k}\right)
$
converges by the ergodicity of $\nu_1$ to 
$\left\langle\nu_1, f\right\rangle:=\int_{\mathcal{Z}}f(z)\nu_1(dz)$ 
and the limit, by the Arzela-Ascoli theorem, will be equal to $F_f^*\left(y_{n_k}\right)$ (as every converging subsequence would have to converge to the limit; which also implies that the subsequential convergence in (\ref{subseq}) is in fact a sequential convergence). The same argument applies for $\eta^{(N)}(f)\left(x_{n_k}\right) \rightarrow\left\langle\nu_2, f\right\rangle=F_f^*\left(x_{n_k}\right)$.

The above would then imply that $\left|F_f^*\left(y_{n_k}\right)-F_f^*\left(x_{n_k}\right)\right| \geq \epsilon$ for every $\left(y_{n_k}, x_{n_k}\right)$. 
This would be a contradiction due to the continuity of $F_f^*$. 
Therefore, the time averages of $f$ under $\nu_1$ and $\nu_2$ will be arbitrarily close to each other. 
Since the class of bounded and Lipschitz continuous functions is separating, it follows that $\nu_1 = \nu_2$.
This establishes the uniqueness of the invariant measure. Let this unique invariant measure be denoted by $\nu$.

(ii)
 Now let us prove that if $z^*$ is an topological aperiodic state, 
 then for any initial state, 
 the Markov process weakly converges to $\nu$. 
 For this, we will show that the approach given in 
 \cite[Theorem 18.4.4-ii]{MeynBook} for $e-$chains
is also valid for $L-$chains. 
\yed{For the continuous-time setting, a different proof technique has been used to establish a similar result for $L$-chains; see Theorem~1 of \cite{komorowski2010ergodicity}.}

First, consider any bounded and Lipschitz continuous function $f$. Without loss of generality, assume that
\(
\int_\mathcal{Z} f(z) \, \nu(dz) = 0,
\)
since we can always subtract a constant from $f$ to ensure this.

By the Arzelà–Ascoli theorem and the equicontinuity of the sequence, the family
\begin{align}\label{seq}
\eta^N(f)(z) := \int \eta^N(d y|z) f(y)
\end{align}
has a subsequence that converges uniformly (in the sup-norm) to a limit function $H_f^* : \mathcal{Z} \rightarrow \mathbb{R}$, where $H_f^*$ is continuous. Therefore, there exists a subsequence $\eta^{k_i}(f)$ that converges uniformly to $H_f^*$.

For every $k \in \mathbb{N}$, we have
\begin{align*}
    \langle \nu,|\eta^k(f)|\rangle&=
    \int_\mathcal{Z} |\eta^k(f)|(z)\nu(dz)=
    \int_\mathcal{Z} \eta(|\eta^k(f)|)(z)\nu(dz)\\
    &=\int_\mathcal{Z}\int_\mathcal{Z} |\eta^k(f)|(y)\eta(d y|z)\nu(dz)
    \geq \int_\mathcal{Z}|\eta^{k+1}(f)(z)|\nu(dz)=
    \langle \nu,|\eta^{k+1}(f)|\rangle
    \end{align*}
where the inequality follows from Jensen's inequality.

By the Monotone Convergence Theorem, the sequence $\langle \nu, |\eta^k(f)| \rangle$ converges to a real limit, denoted by $v$. Since $\eta^{k_i}(f)$ converges uniformly to $H_f^*$, the shifted sequence $\eta^{k_i + m}(f)$ converges to $\eta^m(H_f^*)$ for each $m \in \mathbb{N}$. Hence, we obtain:
\begin{align}\label{H_f_abs}
\langle \nu, |\eta^m(H_f^*)| \rangle = \langle \nu, |H_f^*| \rangle = v, \quad \forall m \in \mathbb{N}.
\end{align}

We now use the topological aperiodicity of $z^*$ to prove that $H_f^*(z) = 0$ for every $z$ in the topological support of $\nu$.

Consider any open neighborhood of $z^*$, denoted by $O$. By the weak Feller property of $\eta$ and the Portmanteau theorem, the function $\eta^k(O|\cdot)$ is lower semi-continuous. Therefore, for sufficiently large $k$, there exists an open neighborhood $\bar{O}$ around $z^*$ such that $\eta^k(O|z) > 0$ for every $z \in \bar{O}$.
Since $z^*$ is reachable, the Markov chain starting from 
$\nu$ visits $\bar{O}$ within a finite time
from almost every initial state. 
This proves that $\nu(O)>0$

 Next, we show that $H_f^*(z)=0$ 
 for all $z$ in the topological support of $\nu$. 
 Assume, by contradiction, that this is not true. Then the sets
 $O^+=\{z:H_f^*(z)>0\}$ and $O^-=\{z:H_f^*(z)<0\}$ 
 both have positive $\nu$-measure, since $\langle \nu, H_f^* \rangle = \langle \nu, f \rangle = 0$. Because $H_f^*$ is continuous, both $O^+$ and $O^-$ are open.
 Due to ergodicity, 
 and similar to above arguments, 
 we know that for any bounded continuous 
 function $h$, $F_h^*(z^*)$ is equal to 
 $\left\langle\nu, h\right\rangle$. 
 Since $\nu(O^+)>0$ and $O^+$ is open, 
 we can define the function $h$ to be $0$ outside $O^+$ 
 and $\left\langle\nu, h\right\rangle>0$. 
 In this case, $F_h^*(z^*)>0$ 
 which proves that there exists a 
 $k_1\in \mathbb{N}$ such that 
 $\eta^{k_1}(O^+|z^*)>0$. Similarly, 
 there exists a $k_2\in \mathbb{N}$ 
 such that $\eta^{k_2}(O^-|z^*)>0$. 
 As $\eta^{k}(O^+|.)$ is lower semi-continuous, 
 there exists an open neighborhood 
 $O^\prime$ around $z^*$ such that for 
 every $z\in O^\prime$, $\eta^{k_1}(O^+|z)>0$ 
 and $\eta^{k_2}(O^-|z)>0$. 
 As $z^*$ is topologically aperiodic, 
 for all sufficiently large $k$, $\eta^{k}(O^\prime|z^*)>0$. 
 Therefore, there exists an $l\in \mathbb{N}$ 
 such that $\eta^{l}(O^+|z^*)>0$ and $\eta^{l}(O^-|z^*)>0$. 
 Again, since $\eta^{l}(O^+|.)$ and $\eta^{l}(O^-|.)$ 
 are lower semi-continuous, 
 there exists an open neighborhood $N$ around 
 $z^*$ such that for every $z\in N$, $\eta^{l}(O^+|z)>0$ 
 and $\eta^{l}(O^-|z)>0$. 
 Therefore, for every $z \in N$, 
 $|\eta^l(H_f^*)(z)|<\eta^l(|H_f^*|)(z)$. 
 Since $z^*$ is in the topological support, 
 we know that $\nu(N)>0$. From here, we reach the conclusion that
$
 \langle \nu,|\eta^l(H_f^*)|\rangle <
 \langle\nu,\eta^l(|H_f^*|)\rangle =
 \langle\nu, |H_f^*|\rangle
$
 where the equality comes from $\nu$ being stationary. 
 This contradicts the equality (\ref{H_f_abs}).

This proves that $H_f^*(z)=0$ for every $z$ 
in the topological support of $\nu$. Therefore, the series 
$\eta^{k_i}(f)$ uniformly converges to $0$
in the topological support of $\nu$. So, for any $\epsilon>0$, 
there exists a $k\in \mathbb{N}$ such that for every 
$z$ in the topological support of $\nu$, $|\eta^{k}(f)(z)|<\epsilon$ holds. 
The set
\(
O^k_\epsilon=\{ z\in \mathcal{Z}: |\eta^k(f)(z)|<\epsilon \}
\)
is an open set and contains all elements in the topological 
support of $\nu$. Therefore, $\nu(O^k_\epsilon)=1$.
For any $m \in \mathbb{N}$ and for any $z$ in the topological support of $\nu$,
\begin{align*}
&|\eta^{k+m}(f)(z)|=
\bigg|\int_\mathcal{Z}\eta^m(dy|z)\eta^k(f)(y) \bigg|=\\
&\bigg|\int_{O^k_\epsilon}\eta^m(dy|z)\eta^k(f)(y) + 
\int_{\mathcal{Z} \setminus O^k_\epsilon}\eta^m(dy|z)\eta^k(f)(y)\bigg|=
\bigg|\int_{O^k_\epsilon}\eta^m(dy|z)\eta^k(f)(y)\bigg| \leq\epsilon.
\end{align*}
This final equality arises because
$z$ is within the topological support of invariant probability measure
$\nu$,
leading to $\eta^m(\mathcal{Z}\setminus O^k_\epsilon|z)=0$.
This shows that $\eta^N(f)$ converges to $0$ 
for every element 
in the topological support of $\nu$.

Next, we extend this convergence to every point $z \in \mathcal{Z}$. 
Define the sequence $\bar{\eta}^N(.|z):=\frac{1}{N}\sum_{i=1}^{N} \eta^i(.|z)$
which converges weakly to the unique invariant measure $\nu$, 
because since $\mathcal{Z}$ is compact, 
the sequence has a limit, and the limit is a stationary probability measure; by uniqueness, it is $\nu$.

For any $\epsilon>0$, the set
\(
O_\epsilon=\{ z\in \mathcal{Z}: \limsup_{k \to \infty}
|\eta^k(f)|<\epsilon \}
\)
is an open set and contains all elements in the topological 
support of $\nu$. 
Therefore, $\nu(O_\epsilon)=1$, 
meaning for any $z\in\mathcal{Z}$, 
we have $\lim_{N \to \infty}\bar{\eta}^N(O_\epsilon|z)=1$. 
From this, for a sufficiently large 
$M\in \mathbb{N}$, $\eta^M(O_\epsilon|z)>1-\epsilon/
\norm{f}_\infty$ holds.

As a result,
\begin{align*}
&|H_f^*(z)|\leq \limsup_{k \to \infty} |\eta^{k+M}(f)(z)|
\leq \norm{f}_\infty \eta^M(O_\epsilon^C|z)+
\limsup_{k \to \infty} \int_{O_\epsilon}
\eta^M(dy|z) |\eta^k(f)(y)|\leq 2\epsilon
\end{align*}

Thus, for every $z\in \mathcal{Z}$, $H_f^*(z)=0$. 
The convergence is uniform,
so the series 
$\eta^{k_i}(f)$ uniformly converges to $0$.
For any $\epsilon>0$, 
there exists a $k\in \mathbb{N}$ such that for every 
$z\in \mathcal{Z}$, $|\eta^{k}(f)(z)|<\epsilon$ holds. 
For any $m \in \mathbb{N}$,
\(
|\eta^{k+m}(f)(z)|=
\bigg|\int_\mathcal{Z}\eta^m(dy|z)\eta^k(f)(y) \bigg|\leq\epsilon
\)

Therefore, $\eta^{N}(f)(z)$ uniformly converges to 
$\langle\nu, f\rangle=0$.
This is valid for all bounded and Lipschitz continuous 
functions and these functions are dense in the set of 
bounded continuous functions.  Consequently, 
for every $z\in \mathcal{Z}$, $\eta^N(.|z)$ weakly converges to $\nu$.
\end{proof}

\begin{proof}[Proof of Theorem \ref{Main2}]    
    Theorem \ref{L-chain} proves that 
    the filter process is an $L$-chain. 
    We also know that 
    $\mathcal{Z}$ is compact.
    We complete the proof using Theorem \ref{Essential}.
\end{proof}

\bibliographystyle{plain}
\bibliography{unique_ergodic}



\end{document}